\documentclass[12pt,reqno]{amsart}
\usepackage{fullpage}
\usepackage{times}
\usepackage{amsmath,amssymb,amsthm,url}
\usepackage[utf8]{inputenc}
\usepackage[english]{babel}
\usepackage{comment}
\usepackage{bbm}
\usepackage{enumerate}

\usepackage{bm}
\usepackage{graphicx}
\usepackage{mathrsfs}
\usepackage[colorlinks=true, pdfstartview=FitH, linkcolor=blue, citecolor=blue, urlcolor=blue]{hyperref}
\usepackage[vlined,ruled]{algorithm2e}
\usepackage{mathtools}
\newtheorem{theorem}{Theorem}[section]
\newtheorem{proposition}[theorem]{Proposition}
\newtheorem{lemma}[theorem]{Lemma}

\newtheorem{corollary}[theorem]{Corollary}

\theoremstyle{definition}
\newtheorem{definition}[theorem]{Definition}
\newtheorem{remark}[theorem]{Remark}

\newcommand{\R}{\mathbb{R}} 
\newcommand{\N}{\mathbb{N}}
\newcommand{\Z}{\mathbb{Z}}

\newcommand{\eps}{\varepsilon}
\title{The Prouhet--Tarry--Escott problem for subsets with small doubling in integral domains}
\author{Ernie Croot}
\address{School of Mathematics\\ Georgia Institute of Technology\\ Atlanta, GA 30332\\ United States}
\email{ernest.croot@math.gatech.edu}
\author{Junzhe Mao}
\address{School of Mathematics\\ Georgia Institute of Technology\\ Atlanta, GA 30332\\ United States}
\email{jmao87@gatech.edu}
\author{Chi Hoi Yip}
\address{Department of Mathematics, Hong Kong University of Science and Technology, Clear Water Bay, Hong Kong}
\email{machyip@ust.hk}
\begin{document}
\begin{abstract}
The Prouhet--Tarry--Escott (PTE) problem has many generalizations and has
been studied in various algebraic domains. In this paper, we prove that finite
subsets $S$ of integral domains with small additive doubling constant (but
still a power of $|S|$) always contain solutions to Wright's generalization
of the PTE problem: there are small subsets $A$ and $B$ of the same size
such that $\sum_{a\in A} a^j=\sum_{b\in B} b^j$ for $1\le j\le k$, but not
for $j=k+1$. More generally, our method gives simultaneous solutions for
$m$ systems, with pairwise distinct $(k+1)$-th power sums. In contrast with
the classical case $S\subseteq [N]$, where the problem has been studied by
Wooley and others using Vinogradov's mean value theorem, our approach is based
on polynomial identities and additive properties of $S$. We also discuss
barriers to extending these results to broader settings.
\end{abstract}

\maketitle

\section{Introduction}

The Prouhet--Tarry--Escott (PTE) problem is that of finding two distinct, non-decreasing sequences of integers 
$a_1,a_2,\ldots,a_s$ and $b_1,b_2,\ldots,b_s$, such that
\begin{equation}\label{multigrade}
\sum_{i=1}^s a_i^j\ =\ \sum_{i=1}^s b_i^j \quad (1\leq j\leq k).
\end{equation}
The PTE problem has a long history and has been studied extensively. There are many well-studied generalizations and variants of the PTE problem. We refer to \cite{W48b, W96} for some historical notes as well as the survey \cite{BI94, B12}.

In this paper, we focus on variants of the PTE problem introduced by Wright. One variant of the PTE problem was introduced by Wright~\cite{W35}, where he looked for solutions to \eqref{multigrade} with the extra condition that $\sum_{i=1}^s a_i^{k+1}\neq \sum_{i=1}^s b_i^{k+1}$. In subsequent work \cite{W48,W48b}, Wright also considered nontrivial solutions to the simultaneous Diophantine equations
\begin{equation}\label{eq:multi_sys}
\sum_{i=1}^sx_{i1}^j = \sum_{i=1}^s x_{i2}^j = \cdots = \sum_{i=1}^s x_{im}^j\quad (1\leq j\leq k).
\end{equation}
Let $P(k,m)$ denote the least $s$ for which equation~\eqref{eq:multi_sys} has an integer solution $\mathbf{x}$ in which the sets $\{x_{1h},\ldots,x_{sh}\}(1\leq h\leq m)$ are distinct. Similarly, let $W(k,m)$ denote the least $s$ such that equation~\eqref{eq:multi_sys} has an integer solution $\mathbf{x}$ with $\sum_{i=1}^s x_{ih}^{k+1}\neq \sum_{i=1}^s x_{it}^{k+1}(h\neq t)$.

It is easy to see $P(k,2)\geq k+1$ and it is an open problem to determine if $P(k,2)=k+1$. It is only known that $P(k,2)=k+1$ when $2\leq k\leq 9$ and $k=11$ \cite{B12}. Using a pigeonhole principle argument one can easily see that $P(k,m)\leq \frac{k(k+1)}{2}+1$. However, the pigeonhole principle does not readily yield an upper bound on $W(k,2)$. Nonetheless, Hua \cite{H38, H49} and Wright \cite{W48b} were able to provide upper bounds on $W(k,m)$ using elementary arguments. 

By considering solutions to equation~\eqref{eq:multi_sys} with $x_{ih}\in [N] := \{1,2,\ldots,N\}$, one can see that estimating $W(k,m)$ is related to Vinogradov's mean value theorem. The best-known upper bound is $W(k,m)\ \leq\ \frac{k(k+1)}{2}+1$, due to Wooley \cite[Theorem 13.1]{W19}.

These results of Hua, Wright, and Wooley, as well as the work of Bourgain--Demeter--Guth \cite{BDG16}, leverage properties of dense subsets of integer intervals $[N] = \{1,2,\ldots,N\}$, as we will discuss in Section \ref{lit_section}.  
However, it is not immediately obvious how to extend these proofs to the case where the solutions are restricted to being contained in sparse subsets of the integers.  In this paper, we address this question for when those sparse subsets have some additive structure, specifically that they have ``small doubling constant" or ``high additive energy"; although our proof gives only an exponential-in-$k$ upper bound on $s$, the number of variables in each system.  

It is worth mentioning that the systems~\eqref{multigrade} and~\eqref{eq:multi_sys} make sense over any ring. In particular, there have been studies on the PTE problem in Gaussian integers, Eisenstein integers, finite fields, function fields, and number fields; see for example \cite{AT07, C13, CMSV24, KP14, P12, W19}. 

In this paper we will also study equation~\eqref{multigrade} over an arbitrary ring (not necessarily commutative) and equation~\eqref{eq:multi_sys} over an integral domain for technical reasons.  Recall that an \emph{integral domain} is a nonzero commutative ring with a multiplicative identity that has no zero divisors. 

\subsection{Main results}
We now state the main theorem of the paper.

\begin{theorem}\label{main_theorem} Let $k\ge 1$, $m\ge 2$ be two integers. There is a constant $c > 0$ depending on $k$ and $m$, such that the following holds. Suppose $R$ is a ring and $S\subseteq R$ is a nonempty finite subset satisfying 
$$
|S-S|\ =\ K |S|,\ {\rm where}\ K \leq c |S|^{1/(2^{k+1}-1)},
$$
then 
\begin{enumerate}[(1)]
\item There exist distinct $A_1, A_2,\ldots,A_m \subseteq S$, $|A_i| = \lceil\log_2m\rceil 2^k$, such that
$$
\sum_{a \in A_i} a^j\ =\ 
\sum_{b \in A_1} b^j \qquad {\rm for\ }j=1,\ldots,k,\ {\rm and\ } i=1,\ldots,m. 
$$
\item Assume additionally that $R$ is an integral domain with characteristic $\mathrm{char}(R)=0$ or $\mathrm{char}(R)>k+1$. Then 
there exist $A_1, A_2,\ldots,A_m \subseteq S$, $|A_i| = \lceil\log_2m\rceil 2^k$, such that
$$
\sum_{a \in A_i} a^j\ =\ 
\sum_{b \in A_1} b^j \qquad 
{\rm for\ }j=1,\ldots,k,\ {\rm and\ } i=1,\ldots,m,
$$
while for all $h,t = 1,2,\ldots, m,\ h\neq t$, we have
$$
\sum_{a\in A_h} a^{k+1}\ \neq\ \sum_{b \in A_t} b^{k+1}.
$$
\end{enumerate}
\end{theorem}

Theorem~\ref{main_theorem} can be applied in a wide range of settings, for example generalized arithmetic progressions, finite fields with large characteristic, polynomials over finite fields with bounded degree, and algebraic integers with bounded norm.

A special case of Theorem~\ref{main_theorem} is the following, which is formulated in a slightly different way.
\begin{theorem}\label{main_theorem_variant}
    Let $k$ be a positive integer. Suppose $R$ is an integral domain with characteristic $\mathrm{char}(R)=0$ or $\mathrm{char}(R)>k+1$. Then there is a positive integer $N = N(k)$ such that for any finite subset $T\subseteq R$ with $|T|\geq N$ and any $S\subseteq T$ of size $|S|\geq (24|T-T|)^{1-1/2^{k+1}}$, there exist $A, B \subseteq S$, $|A| = |B| = 2^k$, such that
$$
\sum_{a \in A} a^j\ =\ 
\sum_{b \in B} b^j,\ {\rm for\ }j=1,\ldots,k,
$$
while 
$$
\sum_{a\in A} a^{k+1}\ \neq\ \sum_{b \in B} b^{k+1}.
$$
\end{theorem}

Using Pl\"unnecke--Ruzsa inequality (see for example~\cite{Pe12}), when $|S+S|\leq K'|S|$, we have $|S-S|\leq K'^2|S|$. This allows us to deduce a version of Theorem~\ref{main_theorem} for sets with small doubling constant. We can also deduce a corollary for sets with large additive energy using the best quantitative version of the Balog--Szemer\'edi--Gowers theorem~\cite{gowers} by Reiher--Schoen~\cite{RS24}:
\begin{corollary}  
The conclusion of Theorem~\ref{main_theorem} remains valid if the hypothesis
on \(|S-S|\) is replaced by the following additive-energy hypothesis:
$$
{\rm E}(S)\ :=\ \#\{a_1, a_2, b_1, b_2 \in S\ :\ 
a_1 + a_2 = b_1 + b_2\}\ >\ c' |S|^{3 - \frac{1}{5(2^{k+1}-1)}},
$$
where $c' > 0$ is some constant depending only on $k$ and $m$.
\end{corollary}

We briefly outline the proof of Theorem~\ref{main_theorem} here. By exploiting the additive properties of $S$, we first find several disjoint large Hilbert cubes in $S$, whose generators satisfy certain algebraic constraints. Then we use polynomial identities to construct solutions based on these Hilbert cubes.

In the classical setting when $R=\Z$ and $T = [N]$, Theorem~\ref{main_theorem_variant} allows us to find a solution to the system of equations with $2^{k+1}$ variables in any subset $S\subseteq [N]$ with $|S|\gg_k N^{1-1/(2^{k+1}-1)}$. As shown by Wooley and others, using Vinogradov's mean value theorem, the bounds on $|S|$ and $|A|,|B|$ can be improved. We refer to Corollary~\ref{wooley_corollary} for a general version.   In contrast, we do not have such powerful tools to estimate the number of solutions to equations~\eqref{multigrade} and~\eqref{eq:multi_sys} in an arbitrary subset of a general ring.

In a different direction, one can ask about lower bounds. For example, given some integers $s$ and $k$, how large must $S \subseteq [N]$ be in order to guarantee that sets $A,B \subseteq S$ with $|A| = |B| \leq s$ exist, such that equation~\eqref{multigrade} holds, but
that the sums do not equal when $j=k+1$?  And what about subsets of general rings? These questions are partially addressed in Section \ref{lower_bound_section}, though the bounds produced are very far from the upper-bound results appearing in Theorem \ref{main_theorem} and Corollary \ref{wooley_corollary} below. We also remark that there might be a way to employ polynomial methods to produce lower bounds. For example, Borwein and Mossinghoff~\cite{BM00} connected a variant of the PTE problem to a question about polynomials with high-order vanishing at $1$.
\medskip

\noindent {\bf Some Additional Remarks:}  
\begin{itemize}
\item First, the assumption on the characteristic of $R$ in Theorem~\ref{main_theorem}(2) cannot be omitted entirely. To see this, consider the case where $R$ has characteristic $p$ and $p \mid (k+1)$. In this case we would have
$$
\sum_{a\in A} a^{k+1} = \left ( \sum_{a\in A} a^{(k+1)/p} \right )^{p} = \left ( \sum_{b \in B} b^{(k+1)/p} \right )^{p} = 
\sum_{b \in B} b^{k+1}.
$$
Thus, the conclusion in Theorem~\ref{main_theorem}(2) cannot hold.

\item Second, we will show in Proposition \ref{propAB} that improving the $2^k$ for the size of the sets $A$ and $B$ in Theorem~\ref{main_theorem_variant} to something closer to $k^2$ would require a different method than ours, or at least would require adding substantial new ingredients.  

\item Third, as we will see in the remark following Corollary \ref{wooley_corollary} below, one cannot, in general, prove such a corollary just assuming the hypotheses of Theorem \ref{main_theorem}.  That is, the ability to pick and choose the exponents $j$ where the 
$\sum_{a \in A} a^j \neq \sum_{b \in B} b^j$ 
is quite limited.  
\end{itemize}

\subsection{Results from the literature and some further directions} \label{lit_section}
Given two positive integers $s,k$, let $J_{s,k}(N)$ denote the number of solutions to equation~\eqref{multigrade} with $a_i,b_i\in [N]$. Then we know
\[J_{s,k}(N) = \int_{\mathbb{T}^k}\bigg|\sum_{1\leq n\leq N}e(\alpha_1n+\cdots+\alpha_kn^k)\bigg|^{2s}d\alpha_1\cdots d\alpha_k,\]
where $\mathbb{T} = \R/\Z$.
Recall that Vinogradov's mean value theorem provides an almost sharp upper bound 
\begin{equation}
J_{s,k}(N)\ll_{s,k,\eps} N^\eps(N^s+N^{2s-k(k+1)/2}).
\end{equation}
In the famous work of Bourgain--Demeter--Guth~\cite{BDG16} and Wooley~\cite{W12, W19}, they established Vinogradov's mean value theorem together with some generalizations. We refer to the survey by Pierce \cite{P19} for further discussion.

In \cite[Corollary 1.4]{W19} Wooley proved the following generalization of Vinogradov's mean value theorem:
 \begin{theorem}[Wooley]\label{thm:Vinogradov}
    Let $s,k\in \N$, $\eps>0$, $(a_n)_{n\in \Z}$ be a sequence of complex numbers. Then 
    \begin{align*}
    &\int_{\mathbb{T}^k}\bigg|\sum_{|n|\leq N}a_ne(\alpha_1n+\cdots+\alpha_kn^k)\bigg|^{2s}d\alpha_1\cdots d\alpha_k \nonumber \\
    &\ll_{s,k,\eps} N^\eps(1+N^{s-k(k+1)/2})\bigg(\sum_{|n|\leq N}|a_n|^2\bigg)^s.    
    \end{align*}
\end{theorem}

Following a simple counting argument (see for example \cite[Section 8]{WIMRN}), one can deduce the following corollary of Theorem~\ref{thm:Vinogradov}.

\begin{corollary}\label{cor:Vinogradov_in_J}
    Suppose that $t \geq 1$ and that $k_1,\ldots,k_t$ are positive integers with $1 \leq k_1 < k_2 < \cdots < k_t = k$.  Let $s \geq k(k+1)/2$, and write $K = k_1 + \cdots + k_t$.  Then for any $\eps > 0$, and any sequence $c_{-N},\ldots, c_{N} \in {\mathbb C}$, one has
\begin{align}\label{integral_eqn}
    &\int_{\mathbb{T}^t}\bigg|\sum_{|n| \leq N}
    c_n e(\alpha_1 n^{k_1} + \cdots + \alpha_t n^{k_t})\bigg|^{2s} d\alpha_1 \cdots d\alpha_t  \ll_{s,k,\eps} N^{s-K+\eps} \left ( \sum_{|n| \leq N} |c_n|^2 \right )^s.
\end{align}
\end{corollary}

As an application, Corollary~\ref{cor:Vinogradov_in_J} will imply the following result (proved in Section~\ref{sec:cor}), providing bounds for a generalization of the function $W(k,m)$. In particular, it generalizes a result of Wooley \cite[Theorem 13.1]{W19}.

\begin{corollary}\label{wooley_corollary}  For every $k \geq 1$ and every proper subset $J \subset \{1,2,\ldots,k\}$, if $N \geq N_0(k)$, then for any integer $m\geq 2$ and $S \subseteq [N]$ satisfying 
$|S| \geq m^{2/k(k+1)}N^{1-2/(k^2+k+1)}$, there exist $A_1,A_2,\ldots,A_m \subseteq S$,
each of size 
$$
s\ =\ k(k+1)/2,
$$
such that
$$
\sum_{a \in A_i} a^j\ =\ \sum_{b \in A_1} b^j,\ {\rm for\ }j \in J,\ {\rm and}\ i = 1,\ldots,m.
$$
while for any $1\leq h < t\leq m$, 
$$
\sum_{a\in A_h}a^{j}\neq \sum_{b\in A_t}b^j,\ {\rm for}\ j \in \{1,2,\ldots,k\}\setminus J.
$$
\end{corollary}

\begin{remark}
The following shows that we cannot get a result like this just assuming $S$ is an arbitrary subset of integers with small doubling constant like in Theorem \ref{main_theorem}:  if we take $S = \{1 + r M\ :\ r=1,2,\ldots,n\}$, then for $M$ large enough, if $A,B \subseteq S$, we have 
$A = \{1 + a_i M\ :\ i=1,\ldots,s\}$ and 
$B = \{1 + b_i M\ :\ i=1,\ldots,s\}$, where the 
$a_i,b_i$ are in $\{1,\ldots,n\}$.  Then, if  
$\sum_{i=1}^s (1 + a_iM)^j \neq \sum_{i=1}^s (1 + b_iM)^j$, expanding out the terms (using the binomial theorem) one sees this implies that for some $h \leq j$,
$\sum_{i=1}^s a_i^h \neq \sum_{i=1}^s b_i^h$; and so, for any $j' \geq h$ we would have to have 
$\sum_{i=1}^s (1 + a_i M)^{j'} \neq \sum_{i=1}^s 
(1 + b_iM)^{j'}$.  Therefore, the only choices of $J$
like in Corollary \ref{wooley_corollary} that would work for this set would be ones of the form
$J = \{1,2,\ldots,m\}$ for some $m \leq k$.

Results similar to Corollary~\ref{wooley_corollary} can also be proved in number fields/function fields as a consequence of Vinogradov's mean value theorem in number fields and function fields proved by Wooley \cite{W19}. The remark above also applies to the number field/function field setting.
\end{remark}

\smallskip

An example consequence of Corollary~\ref{wooley_corollary} is the following corollary related to a theorem of Friedlander and Lagarias on smooth numbers in short intervals
\cite[Theorem 4]{FL87}:

\begin{corollary}\label{smooth_corollary} For integers $r \geq 3$ of the form $k(k+1)/2$, real numbers $0 < \alpha < {1 \over (r+1)\sqrt{2r}}$, integers $N > N_0(\alpha,r)$, and 
sets of integers
$$
C\ \subseteq\ [1, N^{1/r}],
\ |C| > N^{(1 - \alpha)/r},
$$
there exists some constant $\kappa = \kappa(\alpha,r)>0$, an integer $M < N$, and an increasing sequence of integers
$$
x_1,\ x_2,\ \ldots,\ x_n \in [1,2M],\ n\ >\ \kappa M^{\frac{1}{\sqrt{2}r^{5/2}(r+1)} - \frac{\alpha}{r^2}},  
$$
each the product of exactly $r$ numbers from $C$, such that for all $1\leq i \leq n-1$,
$$
x_{i+1} - x_i\ =\ \kappa M^{1- \sqrt{2 \over r} + \varepsilon_i},\ {\rm where\ } \frac{1}{r} \leq \varepsilon_i \leq {3 \over r}.
$$
\end{corollary}

\begin{remark} The proof of \cite[Theorem 4]{FL87}, like the
proof of Corollary \ref{smooth_corollary}, forms products of small integers, up to $N^{1/r}$; however, it does not allow one to choose the set $C$ from which to draw the factors as we do above.  On the other hand, our corollary pays a price by not requiring the $r$-fold products to come close to every integer in $[1,N]$.  

In Section \ref{sec:smooth}, we state and prove a more general corollary (Corollary~\ref{product_difference}) of Corollary~\ref{wooley_corollary} from which it can be derived.
\end{remark}

The results above motivate the following questions:

\begin{itemize}

\item Can one improve the strength of Theorem \ref{thm:Vinogradov} above, so that we get a result like Corollary~\ref{wooley_corollary} just when $|S| \gg_m N^\theta$, for some $\theta \in (0,1)$ that does not depend on $k$?  For example, can one replace $|S| \gg_m N^{1-2/(k^2+k+1)}$ with $|S| \gg_m N^{1/2}$, say?

\item From the remark after Corollary \ref{wooley_corollary} we see that we cannot get a conclusion like in that corollary just assuming the hypotheses of Theorem \ref{main_theorem}.  What kinds of natural conditions could we add to those hypotheses to get such a conclusion?

\item What can one say about the possible sizes of
$\sum_{a \in A} a^j - \sum_{b \in B} b^j$ for 
$j=k+1$ in Theorems \ref{main_theorem} and \ref{main_theorem_variant}, as well as the value of
$j \in \{1,2,\ldots,k\}\setminus J$ in Corollary \ref{wooley_corollary}?  The stronger the conclusion
on this, the tighter the range of differences in Corollary \ref{smooth_corollary} and similar sorts of conclusions.
\end{itemize}

\section{Proof of Theorem \ref{main_theorem}}
To prove the theorem we will need a proposition about Hilbert cubes.  First, though, we need some notation.

\begin{definition}
Given
an abelian group $G$ and
$\alpha_1, \ldots, \alpha_d\in G\setminus \{0\}$, we define
the \emph{Hilbert cube $H_0(\alpha_1, \ldots, \alpha_d)\subseteq G$} to be the set of all sums
$\sum_{s \in S} \alpha_s$ where
$S \subseteq \{1,2,\ldots,d\}$.  Alternatively, we could define this
as the following sumset
$$
H_0(\alpha_1,\ldots,\alpha_d)\ :=\ 
\{0,\alpha_1\} + \{0,\alpha_2\} + 
\cdots + \{0,\alpha_d\}.
$$
In addition, if $\alpha_0$ is some
translate, we let
$$
H(\alpha_0; \alpha_1,\ldots,\alpha_d)
\ :=\ \alpha_0 + H_0(\alpha_1, \ldots, \alpha_d).
$$
We say $H(\alpha_0; \alpha_1,\ldots,\alpha_d)$ is \emph{proper} if $|H(\alpha_0; \alpha_1,\ldots,\alpha_d)| = 2^d$.    
\end{definition}

The following lemma is a variant of Szemer\'edi's cube lemma~\cite{S69}:

\begin{lemma}\label{lem:cube} 
Suppose $d \geq 1$ is an integer, $G$ is an abelian group, 
and that $S$ is a subset of $G$ such that $|S-S| = K |S|$ where
\begin{equation}
|S|\ \geq\ 3^{d+1}(2K)^{2^d-1}.
\end{equation}
Then, $S$ contains a proper Hilbert cube of dimension $d$.
\end{lemma}

\begin{proof}
We construct the Hilbert cube iteratively, by constructing a sequence of elements $\alpha_1,\alpha_2,\ldots,\alpha_d$ in $G$, and a corresponding sequence of sets
$$
S_0\ := S\ \supseteq\ S_1\ \supseteq\ \cdots\ \supseteq\ S_d,
$$
where for  $i \geq 1$ we have that 
$H_0(\alpha_1,\ldots,\alpha_i)$ is proper, and so that
$S_i$ will consist of all the choices
of $\alpha_0$ so that the Hilbert cubes $H(\alpha_0; \alpha_1,\ldots, \alpha_i) \subseteq S$. 

The way we construct a set $S_{i+1}$ from $S_i$ is given as follows:  given $\alpha_1,\ldots,\alpha_i$, among all choices of $x\in S_i-S_i$, we initially select those so that 
$H_0(\alpha_1,\ldots,\alpha_i,x)$ is proper; and then among {\it those}, we select $\alpha_{i+1} := x$ so that $|S_i \cap (S_i - x)|$ is maximal.  We furthermore will 
show that the sets $S_j$, $j \geq 0$, we construct have the property that
\begin{equation}\label{Kip}
|S_j - S_j|\ =\ K_j |S_j|,\ {\rm where\ } K_j\ \leq\ {(2K)^{2^j} \over 2},\ {\rm and\ } |S_j| \geq 3^{d+1}.
\end{equation}
Note that $K_0 = K$.

Let us first establish there are many choices for $x$ and that we can get that $S_i \cap (S_i-x)$ has an abundance of elements for at least one of them:  we assume that $H_0(\alpha_1,\ldots,\alpha_i)$ is proper.  We will say that a choice for $x$ is 
{\it bad} if $H_0(\alpha_1,\ldots,\alpha_i, x)$ is not proper.
Notice that since 
\[|H_0(\alpha_1,\ldots,\alpha_i)-H_0(\alpha_1,\ldots,\alpha_i)| \leq 3^i,\]
there can be at most $3^{i}$ bad choices for 
$x$. Let $X$ denote the remaining {\it good} choices for
$x \in S_i-S_i$ (note that this means $S_i \cap (S_i - x)\neq \emptyset$).  
Obviously 
$$
|X|\ \geq\ |S_i - S_i| - 3^d.
$$
This lower bound is trivially non-zero assuming the last lower bound in \eqref{Kip}.

We have that
\begin{eqnarray}\label{maxx}
\max_{x \in X} |S_i \cap (S_i - x)|
\ &\geq&\ {1 \over |X|} 
\sum_{x \in X} \#\{a, b \in S_i\ :\ a-b = x\}\nonumber \\
&\geq&\ {|S_i|^2 - 3^i|S_i| \over |S_i - S_i|}\ \geq\ 
{|S_i| \over 2 K_i}.
\end{eqnarray}
As we said before, we let $\alpha_{i+1}$ denote this maximal choice for $x \in X$.

If we then let $S_{i+1}$ denote this set $S_i \cap (S_i - \alpha_{i+1})$ we have
that
$$
K_{i+1}\ =\ {|S_{i+1} - S_{i+1}| \over |S_{i+1}|}\ \leq\ {|S_i - S_i| \over
|S_i|/2K_i}\ =\ {K_i |S_i| \over 
|S_i|/2K_i}\ =\ 2 K_i^2,
$$
provided $|S_i| \geq 2$.  So,
$$
K_{i+1}\ \leq\ 2 K_i^2\ \leq\ 
2^3 K_{i-1}^4\ \leq\ 2^7 K_{i-2}^8
\ \leq\ \cdots\ \leq\ 2^{2^{i+1}-1} K_0^{2^{i+1}}\ =\ {(2K)^{2^{i+1}} \over 2}.
$$

From inequality~\eqref{maxx}, it follows that
$$
|S_{i+1}|\ \geq\ {|S_i| \over 2K_i}
\ \geq\ {|S_{i-1}| \over 4 K_i K_{i-1}}
\ \geq\ \cdots\ \geq\ {|S| \over 2^{i+1} K_i K_{i-1}\cdots K_0}\ \geq\ 
{|S| \over (2K)^{2^{i+1}-1}}.
$$
We note that if $i \leq d-1$ then
the last inequality of \eqref{Kip} holds, provided
$$
|S|\ \geq\  3^{d+1}(2K)^{2^d-1}.
$$

To finish the proof, we need to establish that \(S_d\) contains all the choices
for \(\alpha_0\) such that \(H(\alpha_0;\alpha_1,\ldots,\alpha_d)\subseteq S\).
We prove this by induction. Assume that \(S_i\) consists of all choices of
\(\alpha_0\) such that \(H(\alpha_0;\alpha_1,\ldots,\alpha_i)\subseteq S\).
Now suppose
\[
        \alpha'_0\in S_{i+1}=S_i\cap(S_i-\alpha_{i+1}).
\]
Then \(\alpha'_0\in S_i\) and \(\alpha'_0+\alpha_{i+1}\in S_i\). By the
induction hypothesis,
\[
        H(\alpha'_0;\alpha_1,\ldots,\alpha_i)\subseteq S
        \quad\text{and}\quad
        H(\alpha'_0+\alpha_{i+1};\alpha_1,\ldots,\alpha_i)\subseteq S.
\]
Since
\[
        H(\alpha'_0;\alpha_1,\ldots,\alpha_i)\cup        H(\alpha'_0+\alpha_{i+1};\alpha_1,\ldots,\alpha_i)
        =
        H(\alpha'_0;\alpha_1,\ldots,\alpha_{i+1}),
\]
we get \(H(\alpha'_0;\alpha_1,\ldots,\alpha_{i+1})\subseteq S\). This completes
the induction and finishes the proof of the lemma.
\end{proof}

To prove Theorem~\ref{main_theorem}, we need to find several Hilbert cubes in the set $S$, with their generators satisfying some algebraic relations. This can be summarized in the following proposition.

\begin{proposition}\label{prop:cube}
    Suppose $d,\ell$ are positive integers, $R$ is an integral domain, and $S\subseteq R$ is finite such that $|S-S| \leq \ell K|S|$, where
    \begin{equation}\label{eq:Slower}
    |S| \geq 2\ell(3^{d+1}+3^{\ell})(4\ell^2K)^{2^d-1}
    \end{equation}
    Then there exist $\ell$ disjoint proper Hilbert cubes $H_i = H(\alpha_{i,0};\alpha_{i,1},\ldots,\alpha_{i,d}),(1\leq i\leq \ell)$, such that $H_i\subseteq S$ for all $i$, and
    \begin{equation}\label{eq:gener_relation}
    \sum_{i=1}^\ell \eps(i)\prod_{j=1}^d\alpha_{i,j}\neq 0
    \end{equation}
    for any $\eps\in\{-1,0,1\}^\ell\setminus\{(0,0,\ldots,0)\}$.
\end{proposition}

\begin{proof}
    We partition $S$ into $\ell$ disjoint subsets $S_1,S_2,\ldots,S_\ell$ with $|S_i|\geq \lfloor|S|/\ell\rfloor$ for all $i$. Then for any fixed $1\leq i\leq \ell$, we have
    \begin{equation}\label{eq:doubling}
    |S_i-S_i|\leq |S-S|\leq \ell K|S|\leq 2\ell^2 K|S_i|.
    \end{equation}
    We also know that 
    \begin{equation}\label{eq:part_lower}
    |S_i|\geq \frac{|S|}{2\ell}\geq (3^{d+1}+3^\ell)(4\ell^2K)^{2^d-1},
    \end{equation}
    hence by Lemma~\ref{lem:cube} there exists a proper Hilbert cube
    $H_1 = H(\alpha_{1,0};\alpha_{1,1},\ldots,\alpha_{1,d})\subseteq S_1$. 

    To guarantee the algebraic restriction~\eqref{eq:gener_relation}, we construct the remaining Hilbert cubes inductively. Suppose that we have constructed a proper Hilbert cube
    \[
        H_i = H(\alpha_{i,0};\alpha_{i,1},\ldots,\alpha_{i,d})\subseteq S_i
    \]
    for every $1\leq i\leq t$, and that
    \[
        \sum_{i=1}^t \eps(i)\prod_{j=1}^d\alpha_{i,j}\neq 0
    \]
    for every nonzero $\eps\in \{-1,0,1\}^t$. For $i=t+1$, following the proof of Lemma~\ref{lem:cube} for \(d-1\) steps, we can construct
    \[
        A_0:=S_{t+1}\supseteq A_1\supseteq\cdots\supseteq A_{d-1}
    \]
    and $\alpha_{t+1,1},\ldots,\alpha_{t+1,d-1}$ such that
    \[
        |A_j-A_j|=K_j|A_j|,\qquad
        K_j\leq {(4\ell^2K)^{2^j}\over 2},\qquad
        |A_j|\geq 3^{d+1}+3^\ell .
    \]
    The last inequality follows from
    \[
        |A_j|\geq \frac{|S_{t+1}|}{(4\ell^2K)^{2^j-1}}
    \]
    and inequality~\eqref{eq:part_lower}. Put \(B=A_{d-1}\). Since \(B\) is the set of admissible translates for the \((d-1)\)-dimensional cube, to add the last generator we must choose
    \[
        \alpha_{t+1,d}\in B-B,
    \]
    so that some translate \(\alpha_{t+1,0}\in B\cap(B-\alpha_{t+1,d})\) exists.

    There are at most \(3^{d-1}\) choices of \(x\in B-B\) for which
    \[
        H_0(\alpha_{t+1,1},\ldots,\alpha_{t+1,d-1},x)
    \]
    is not proper. Also, for any $\eps\in\{-1,0,1\}^t\times\{\pm1\}$, the equation
    \[
        \sum_{i=1}^t \eps(i)\prod_{j=1}^d\alpha_{i,j}
        +\eps(t+1)x\prod_{j=1}^{d-1}\alpha_{t+1,j}=0
    \]
    has at most one solution in \(x\), since \(R\) is an integral domain and
    \(\prod_{j=1}^{d-1}\alpha_{t+1,j}\neq 0\), with the usual convention that an
empty product is \(1\). Thus there are at most \(3^\ell\)
    additional forbidden choices. Since
    \[
        |B-B|\geq |B|\geq 3^{d+1}+3^\ell>3^{d-1}+3^\ell,
    \]
    we can choose \(\alpha_{t+1,d}\in B-B\) avoiding all forbidden choices, and then choose
    \[
        \alpha_{t+1,0}\in B\cap(B-\alpha_{t+1,d}).
    \]
    This gives us a proper Hilbert cube
    \[
        H_{t+1}=H(\alpha_{t+1,0};\alpha_{t+1,1},\ldots,\alpha_{t+1,d})
        \subseteq S_{t+1}
    \]
    so that
    \[
        \sum_{i=1}^{t+1}\eps(i)\prod_{j=1}^d\alpha_{i,j}\neq 0
    \]
    for every nonzero $\eps\in \{-1,0,1\}^{t+1}$. This finishes the induction step.
\end{proof}

We will also use the following polynomial identity.

\begin{lemma}\label{lem:parity_identity}
Let $d\geq 1$. In the free associative algebra
$\Z\langle x_0,x_1,\ldots,x_d\rangle$, we have
\begin{equation}\label{eq:parity_lower}
\sum_{U\subseteq [d]}(-1)^{|U|}
\left(x_0+\sum_{i\in U}x_i\right)^j=0
\qquad (0\leq j<d).
\end{equation}
Moreover, in the commutative polynomial ring
$\Z[x_0,x_1,\ldots,x_d]$, we have
\begin{equation}\label{eq:parity_top}
\sum_{U\subseteq [d]}(-1)^{|U|}
\left(x_0+\sum_{i\in U}x_i\right)^d
=
(-1)^d d!x_1\cdots x_d.
\end{equation}
\end{lemma}

\begin{proof}
To prove equation~\eqref{eq:parity_lower}, fix a non-commutative monomial $W$ of degree $j<d$, and let
\[
T=\{i\in [d]:x_i\text{ occurs in }W\}.
\]
The coefficient of $W$ in $(x_0+\sum_{i\in U}x_i)^j$ is $1$ if $T\subseteq U$, and is $0$ otherwise. Since $|T|\leq j<d$, the coefficient of $W$ in the left-hand side of equation~\eqref{eq:parity_lower} is
\[
\sum_{[d] \supseteq U\supseteq T}(-1)^{|U|}
=
(-1)^{|T|}(1-1)^{d-|T|}
=
0.
\]
This proves equation~\eqref{eq:parity_lower}.

For equation~\eqref{eq:parity_top}, the same argument shows that every monomial of degree $d$ missing some $x_i$, $1\leq i\leq d$, has coefficient zero. The only remaining monomial is $x_1\cdots x_d$, whose coefficient is $(-1)^d d!$.
\end{proof}

Now we are ready to present the proof of Theorem~\ref{main_theorem}.

\begin{proof}[Proof of Theorem~\ref{main_theorem}(2)]
Given Proposition~\ref{prop:cube}, to prove our theorem we first let $\ell = \lceil\log_2 m\rceil$, $d=k+1$, and
then we note that~\eqref{eq:Slower} 
would hold, provided that 
$$
K\ \leq\ c|S|^{1/(2^{k+1}-1)},\ {\rm where\ } c\ =\ \frac{1}{4\ell[2\ell(3^{k+2}+3^{\ell})]^{1/(2^{k+1}-1)}}.
$$
Note that $c$ here depends only on $k$ and $m$.  So, assuming $S$ satisfies the hypotheses of the main theorem, we get that it contains $\ell$ disjoint proper Hilbert cubes $H_i = H(\alpha_{i,0};\alpha_{i,1},\ldots,\alpha_{i,k+1})$, $(1\leq i\leq \ell)$, each of dimension $k+1$ as given by the above proposition.  

The idea of the rest of the proof is to prove a polynomial analogue of the main theorem, and then to replace the variables with the $\alpha_{i,j}$'s.  Towards this goal we define a set of special linear forms (with coefficients $0$ and $1$) in ${\mathbb Z}[x_1,\ldots,x_{k+1}]$ as follows:
\[
\mathcal{ L}_{k+1}\ :=\ \{ \varepsilon_1 x_1 + \cdots + \varepsilon_{k+1} x_{k+1}\ :\ 
{\rm for\ }i=1,\ldots,k+1,\ \varepsilon_i \in \{0,1\}\}.
\]
We then let $\mathcal{ O}_{k+1}$ denote the elements of $\mathcal{ L}_{k+1}$ with an odd number of non-zero terms, and $\mathcal{ E}_{k+1}$ denote the elements with an even number of non-zero terms. By Lemma~\ref{lem:parity_identity}, for $j=1,\ldots,k$,
\begin{equation}\label{foe}
\sum_{f\in \mathcal{ O}_{k+1}}(x_0+f(x_1,\ldots,x_{k+1}))^j
=
\sum_{g\in \mathcal{ E}_{k+1}}(x_0+g(x_1,\ldots,x_{k+1}))^j,
\end{equation}
while
\begin{equation}\label{eq:parity_k+1}
\sum_{f\in \mathcal{ O}_{k+1}}(x_0+f(x_1,\ldots,x_{k+1}))^{k+1}
-
\sum_{g\in \mathcal{ E}_{k+1}}(x_0+g(x_1,\ldots,x_{k+1}))^{k+1}
=
(-1)^k(k+1)!x_1\cdots x_{k+1}.
\end{equation}

Now for any fixed $1\leq i\leq \ell$, replacing $x_r$ by $\alpha_{i,r}$ in equations~\eqref{foe} and~\eqref{eq:parity_k+1}, we have
\[
\sum_{a\in B_{i,0}}a^j = \sum_{b\in B_{i,1}}b^j
\]
for all $1\leq j\leq k$ while
\begin{equation}\label{eq:k+1_diff}
\sum_{a\in B_{i,0}}a^{k+1}-\sum_{b\in B_{i,1}}b^{k+1} = (-1)^k(k+1)!\alpha_{i,1}\cdots\alpha_{i,k+1}
\ \neq\ 0,
\end{equation}
where 
\[B_{i,0} = \{\alpha_{i,0}+f(\alpha_{i,1},\ldots,\alpha_{i,k+1}):f\in \mathcal{O}_{k+1}\},\quad B_{i,1} = \{\alpha_{i,0}+g(\alpha_{i,1},\ldots,\alpha_{i,k+1}):g\in \mathcal{E}_{k+1}\}.\]
Here the last inequality follows since $\mathrm{char}(R)>k+1$ or $\mathrm{char}(R)=0$, and $\alpha_{i,1}\cdots\alpha_{i,k+1}$ is non-zero. Since $H_i$ is proper, we also have $|B_{i,0}|=|B_{i,1}|=2^k$.

Next we define $2^\ell$ distinct subsets of $S$ as follows: for each $\sigma\in \{0,1\}^\ell$, let 
\[A_\sigma = \bigcup_{i=1}^\ell B_{i,\sigma(i)}.\]
For any $\sigma\neq \tau\in \{0,1\}^\ell$, from our definition of $B_{i,0},B_{i,1}$ we must have
\[\sum_{a\in A_\sigma}a^j = \sum_{b\in A_\tau}b^j\]
for any $1\leq j\leq k$. On the other hand, from equation~\eqref{eq:k+1_diff}, there exists some $\eps\in \{-1,0,1\}^\ell\setminus\{(0,0,\ldots,0)\}$ such that
\[\sum_{a\in A_\sigma}a^{k+1}-\sum_{b\in A_\tau}b^{k+1} = (k+1)!\sum_{i=1}^\ell \eps(i)\prod_{j=1}^{k+1}\alpha_{i,j}.\]
By Proposition~\ref{prop:cube}, this is nonzero. Since $|A_\sigma| = \sum_{i}|B_{i,\sigma(i)}| = \ell 2^k$ for each $\sigma\in \{0,1\}^\ell$, and $2^\ell\geq m$, this finishes the proof of the theorem.
\end{proof}

\medskip

\begin{proof}[Proof of Theorem~\ref{main_theorem}(1)]
The proof is almost identical to the proof above. We shall only need the existence of $\ell = \lceil\log_2m\rceil$ disjoint proper Hilbert cubes $H_1,\ldots, H_\ell$ in $S$, each of dimension $k+1$, without the extra algebraic constraint~\eqref{eq:gener_relation}. This can be proved in any ring by modifying the proof of Proposition~\ref{prop:cube}. Also, we now regard $\mathcal{L}_{k+1}$ as a subset of $\Z\langle x_1,\ldots,x_{k+1}\rangle$, polynomials in non-commuting variables. By equation~\eqref{eq:parity_lower}, for $j=1,\ldots,k$,
\[
\sum_{f \in \mathcal{O}_{k+1}} (x_0+f(x_{1},\ldots,x_{k+1}))^{j}
= \sum_{g \in \mathcal{E}_{k+1}} 
(x_{0}+g(x_{1},\ldots,x_{k+1}))^{j}
\]
holds in $\Z\langle x_0,x_1,\ldots,x_{k+1}\rangle$.

It follows that for any $1\leq i\leq \ell$, $1\leq j\leq k$,
\[
\sum_{f \in \mathcal{O}_{k+1}} (\alpha_{i,0}+f(\alpha_{i,1},\ldots,\alpha_{i,k+1}))^{j}
= \sum_{g \in \mathcal{E}_{k+1}} 
(\alpha_{i,0}+g(\alpha_{i,1},\ldots,\alpha_{i,k+1}))^{j}.
\]
Defining $B_{i,0},B_{i,1}$ for $1\leq i\leq \ell$ and $A_\sigma$ for $\sigma\in \{0,1\}^\ell$ as in the proof of Theorem~\ref{main_theorem}(2) completes the proof.
\end{proof}

We have the following proposition showing that for the polynomial analogue established in the above proof, the choice of $\mathcal{O}_{k+1}$ and $\mathcal{E}_{k+1}$ is optimal.
\begin{proposition}\label{propAB}
Let \(h=k+1\). For each \(U\subseteq [h]\), write
\[
        \ell_U=\sum_{i\in U}x_i\in \mathcal L_h .
\]
Let \(\mathcal A,\mathcal B\subseteq \mathcal L_h\) be disjoint and nonempty, and suppose that
\(\ell_\emptyset=0\in \mathcal A\cup \mathcal B\). Assume that
\[
        \sum_{\ell_U\in\mathcal A}\ell_U^j
        =
        \sum_{\ell_U\in\mathcal B}\ell_U^j
\]
as identities in \(\mathbb Z[x_1,\ldots,x_h]\) for $1\leq j\leq k$. Then
\[
        \mathcal{A}\cup \mathcal{B}=\mathcal L_h.
\] Moreover, after possibly interchanging \(\mathcal A\) and \(\mathcal B\), the nonzero elements
of \(\mathcal A\) are precisely the \(\ell_U\) with \(|U|\equiv h\pmod 2\), and the nonzero
elements of \(\mathcal B\) are precisely the \(\ell_U\) with \(|U|\not\equiv h\pmod 2\). The zero
form may lie on either side. \end{proposition}
\begin{proof}
For \(U\subseteq [h]\), set
\[
        w_U=\mathbf 1_{\ell_U\in \mathcal{A}}-\mathbf 1_{\ell_U\in \mathcal{B}},
        \qquad
        F(T)=\sum_{U\supseteq T}w_U \quad (T\subseteq [h]).
\]
We claim that \(F(T)=0\) for every nonempty proper subset \(T\subsetneq [h]\).
Indeed, writing \(r=|T|\le h-1=k\), the coefficient of
\[
        x_T=\prod_{i\in T}x_i
\]
in \(\ell_U^r\) is \(r!\) if \(T\subseteq U\), and is \(0\) otherwise. Comparing
the coefficient of \(x_T\) in the identity for the \(r\)-th powers gives
\[
        0=r!\sum_{U\supseteq T}w_U=r!F(T).
\]

By Möbius inversion on the subset lattice \(2^{[h]}\) (see for example
\cite[Example~3.8.3]{S97}),
\[
        w_U=\sum_{T\supseteq U}(-1)^{|T|-|U|}F(T)
        \qquad(U\subseteq [h]).
\]
Thus, for every nonempty \(U\subseteq [h]\), all terms vanish except possibly
\(T=[h]\), and so
\[
        w_U=(-1)^{h-|U|}F([h]).
\]
Since \(\mathcal{A}\) and \(\mathcal{B}\) are nonempty and disjoint, and since
\(\ell_\emptyset\in \mathcal{A}\cup \mathcal{B}\), some nonzero \(\ell_U\) belongs to \(\mathcal{A}\cup \mathcal{B}\).
Hence \(w_U\ne0\) for some nonempty \(U\), so \(F([h])\ne0\). But
\(F([h])=w_{[h]}\in\{-1,0,1\}\), hence \(F([h])=\pm1\). Therefore
\(w_U\in\{\pm1\}\) for every nonempty \(U\), so every nonzero element of
\(\mathcal L_h\) lies in exactly one of \(\mathcal{A},\mathcal{B}\). Together with
\(\ell_\emptyset\in \mathcal{A}\cup \mathcal{B}\), this gives \(\mathcal{A}\cup \mathcal{B}=\mathcal L_h\).

The same formula also gives the stated parity description after possibly
interchanging \(\mathcal{A}\) and \(\mathcal{B}\). The zero form is not determined by the identities,
since all of its positive powers vanish.
\end{proof}

\section{Lower bounds} \label{lower_bound_section}

In the setting of Theorem~\ref{main_theorem_variant}, one would expect that when $S$ becomes sparser, more variables ($|A|+|B|$) are needed to guarantee the existence of a solution. When $R = \Z$ and $T = [N]$, given some positive integer $M = M(N)$ allowed to grow with $N$, and let $k\geq 2$ be fixed, we would like to know what is the smallest number $s$ such that for any $S\subseteq[N]$ of size $|S| = M$, there exist $A,B\subseteq S$ with $|A|=|B|=s$ such that 
\begin{equation}\label{eq:tarry}
\sum_{a \in A} a^j\ =\ \sum_{b \in B} b^j,\ {\rm for\ }1\leq j\leq k,\ {\rm but\ not\ }k+1 .
\end{equation}
It turns out that this problem is related to the study of generalized Sidon sets. Given an abelian group $G$ and an integer \(s\ge2\), we call a set \(A\subseteq G\) to be a \emph{$B_s[1]$ set} if for any $g\in G$, the number of solutions to the equation
\[
g = x_1+x_2+\cdots+x_s
\]
with $x_1,\ldots,x_s\in A$ is at most $1$, where we consider two such solutions to be the same if they differ only in the ordering of the summands. In particular, Sidon sets are precisely $B_2[1]$ sets. When $S\subseteq [N]$ is chosen in such a way that $\tilde{S}:= \{(n,n^2,\ldots,n^k):n\in S\}\subseteq \Z^k$ is a $B_s[1]$ set, it is impossible to find $A,B\subseteq S$ with $|A|=|B|=s$ satisfying \eqref{eq:tarry}. Notice that the discrete moment curve $\mathcal{C}_N^k = \{(n,n^2,\ldots,n^k):n\in [N]\}$ is a $B_{k}[1]$ set. This shows that $s$ must be greater than $g(M)$, where
\[
g(M) = \max\{s\in\N: \text{there exists a }B_s[1] \text{ subset of }\mathcal{C}_N^k \text{ of size }M\}.
\]
Given $s>k$, the asymptotic size of the maximum $B_s[1]$ subset of $\mathcal{C}_N^k$ seems still unknown. Here we provide a probabilistic construction that produces a large $B_s[1]$ subset of $\mathcal{C}_N^k$ when $s>k$ and $N$ is sufficiently large, but we do not expect such a construction to be optimal.

\begin{proposition}\label{prop:Bs1}
    Suppose $k,s$ are positive integers such that $k\geq 2$ and $s > k$. Then for any $\eps>0$, when $N>N_0(k,s,\eps)$, there exists a $B_s[1]$ subset of $\mathcal{C}_N^k$ of size $\gg_{k,s,\eps} N^{\frac{k-2\eps}{2(s-1)}}$.
\end{proposition}

We first prove the following lemma, which allows us to bound the number of solutions to equation~\eqref{multigrade} with repeated elements. 

\begin{lemma}\label{lem:mixed-holder}
Let
\[
        f(\alpha)=\sum_{1\le n\le N} e(\alpha_1 n+\alpha_2 n^2+\cdots+\alpha_k n^k),
        \qquad \alpha=(\alpha_1,\dots,\alpha_k)\in [0,1)^k .
\]
For an integer \(c\ne 0\), write
\[
        c\alpha=(c\alpha_1,\dots,c\alpha_k)\pmod 1.
\]
Let \(n_1,\dots,n_t,m_1,\dots,m_r\) be nonzero integers, and define
\[
        I=
        \int_{[0,1)^k}
        \prod_{i=1}^t f(n_i\alpha)
        \prod_{j=1}^r \overline{f(m_j\alpha)}
        \,d\alpha .
\]
If \(L=t+r\), then
\[
        |I|
        \le
        J_{\lfloor L/2\rfloor,k}(N)^{1/2}
        J_{\lceil L/2\rceil,k}(N)^{1/2}.
\]
\end{lemma}

\begin{proof}
By H\"older's inequality with exponent \(L\),
\[
        |I|
        \le
        \prod_{i=1}^t \|f(n_i\cdot)\|_L
        \prod_{j=1}^r \|f(m_j\cdot)\|_L .
\]
Since multiplication by any nonzero integer \(c\) is measure-preserving on the torus \([0,1)^k\),
we have
\[
        \|f(c\cdot)\|_L=\|f\|_L .
\]
Hence
\[
        |I|
        \le
        \|f\|_L^L
        =
        \int_{[0,1)^k}|f(\alpha)|^L\,d\alpha .
\]
Put \(u=\lfloor L/2\rfloor\) and \(v=\lceil L/2\rceil\). Thus, by Cauchy--Schwarz, we have
\[        
|I|
        \le
        \left(\int_{[0,1)^k}|f(\alpha)|^{2u}\,d\alpha\right)^{1/2}
        \left(\int_{[0,1)^k}|f(\alpha)|^{2v}\,d\alpha\right)^{1/2}
        =J_{u,k}(N)^{1/2} J_{v,k}(N)^{1/2}. \qedhere
\]
\end{proof}

\begin{proof}[Proof of Proposition~\ref{prop:Bs1}]
Let $S\subseteq [N]$ be a random set formed by picking each element independently with probability $p$.
For any $k< L\leq 2s$, we need to bound the number of nontrivial solutions to \eqref{multigrade} in $[N]$ with exactly $L$ distinct elements, denoted by $T_{L,s}(N)$. (Using the Vandermonde matrix, it is easy to see that when $L\leq k$, all solutions to \eqref{multigrade} are trivial.) For each such solution, if there is some $a_i = b_{i'}$, we delete $a_i$ and $b_{i'}$. Continue this process until the remaining sequence $a_{i_1},\ldots,a_{i_t}$ is disjoint from $b_{i_1'},\ldots,b_{i_t'}$, and let $L'$ denote the number of distinct elements in $\{a_{i_1},\ldots,a_{i_t}\}\cup\{b_{i_1'},\ldots,b_{i_t'}\}$. It follows from Lemma~\ref{lem:mixed-holder} that  
\[
T_{L,s}(N) = O_{k,s}\left(\sum_{\max(k,L-s+k)\leq L'\leq L}J_{\lfloor L'/2\rfloor,k}(N)^{1/2}J_{\lceil L'/2\rceil,k}(N)^{1/2}N^{L-L'}\right).
\]
Fix some $\eps>0$, Vinogradov's mean value theorem implies that 
\[
J_{\lfloor L'/2\rfloor,k}(N)^{1/2}J_{\lceil L'/2\rceil,k}(N)^{1/2} = \begin{cases}
    O_{k,s,\eps}(N^{L'/2+\eps}), & \text{if }2\leq L'\leq k(k+1)\\
    O_{k,s,\eps}(N^{L'-k(k+1)/2+\eps}), & \text{if }k(k+1)<L'\leq 2s.
\end{cases}
\]
Hence we have 
$$
T_{L,s}(N) =\begin{cases}
     O_{k,s,\eps}(N^{L-k/2+\eps}),\ &\text{if }L\leq s\\
     O_{k,s,\eps}(N^{(L+s-k)/2+\eps}),\ &\text{if }s<L\leq s+k^2\\
     O_{k,s,\eps}(N^{L-k(k+1)/2+\eps}),\ &\text{if }s+k^2<L\leq 2s
\end{cases}
$$
Notice that if $s\leq k^2$, only the first two cases are possible.
Since $s>k$, by taking $p = \delta N^{\frac{1-s+k/2-\eps}{s-1}}<1$ for some sufficiently small constant $\delta = \delta(k,s,\eps) >0$, when $N>N_0(k,s,\eps)$, one can verify that the expected number of nontrivial solutions in $S$ is 
\[
\sum_{L=k}^{2s}p^LT_{L,s}(N) = O_{k,s,\eps}\left(p^sN^{s-k/2+\eps} \right) < pN/2.
\]
Now we can apply the alteration method to get a subset of size at least $pN/2$ that contains no nontrivial solutions. Hence there exists a $B_s[1]$ subset of $\mathcal{C}_N^k$ of size 
\[
\frac{pN}{2} = \frac{\delta}{2}N^{\frac{k-2\eps}{2(s-1)}}.\qedhere
\]
\end{proof}
Therefore, given $M = N^c$ for some $0<c<\frac{k}{2(k-1)}$, when $N$ is sufficiently large, we have 
\[g(M)\geq \bigg\lfloor\frac{k+2c-2\eps}{2c}\bigg\rfloor = \bigg\lfloor1+\frac{(k-2\eps)\log N}{2\log M}\bigg\rfloor.\]

This shows that, in the interval setting, one cannot force such a conclusion
for all subsets of size \(N^{O(k/2^k)}\) with the same number of variables.

\begin{remark}
In the general setting, one can still obtain a universal lower bound by a greedy
argument. Let $R$ be an integral domain and \(T\subseteq R\) be finite, and write
\(\mathcal C_T^k=\{(n,n^2,\ldots,n^k):n\in T\}\subseteq R^k\).
Let \(H=\infty\) if \(\mathrm{char}(R)=0\), and
\(H=\mathrm{char}(R)-1\) otherwise. For \(M\le |T|\), put
\[
g_T(M)=
\max\{h\in\mathbb N:h\le H,\ 
\text{there exists a }B_h[1]\text{ subset of }\mathcal C_T^k
\text{ of size }M\}.
\]
In general, good estimates for the number of solutions to \eqref{multigrade}
in \(T\) may not be available. Nevertheless, a standard greedy argument gives a
\(B_h[1]\) subset of \(\mathcal C_T^k\) of size
\(\gg_h |T|^{1/(2h-1)}\) for every positive integer \(h\le H\). Consequently, if
\(M=|T|^c\) for some \(0<c<1\), then, for \(|T|\) sufficiently large,
\[
g_T(M)\ge
\min\left\{H,\left\lceil\frac{c+1}{2c}\right\rceil-1\right\}.
\]
\end{remark}

\section{Proof of Corollaries~\ref{wooley_corollary} and~\ref{smooth_corollary}}\label{sec:cor}
\subsection{Proof of Corollary \ref{wooley_corollary}}

We first prove a combinatorial lemma.

\begin{lemma}\label{lem:coordinate-packing}
Let \(\mathcal A\subseteq\prod_{j\in\Lambda}\Omega_j\), where \(\Lambda\) and
the \(\Omega_j\)'s are finite nonempty sets. If \(\mathcal A\) contains no \(m\) elements
\(\mathbf u_1,\ldots,\mathbf u_m\) satisfying
\[
        u_{h,j}\ne u_{t,j}\qquad(1\le h<t\le m,\ j\in\Lambda),
\]
then
\[
        |\mathcal A|
        \le
        (m-1)\sum_{j\in\Lambda}
        \prod_{\substack{r\in\Lambda\\ r\ne j}}|\Omega_r|.
\]
\end{lemma}

\begin{proof}
Choose a maximal coordinatewise separated collection
\(\mathbf u_1,\ldots,\mathbf u_q\in\mathcal A\). Then \(q\le m-1\), and by
maximality every element of \(\mathcal A\) shares a coordinate with some
\(\mathbf u_\ell\). For each fixed \(\ell\) and \(j\), at most
\(\prod_{r\ne j}|\Omega_r|\) elements of the ambient product have \(j\)-th
coordinate \(u_{\ell,j}\). Summing over \(\ell\le q\) and \(j\in\Lambda\)
gives the result.
\end{proof}

Next we prove Corollary~\ref{wooley_corollary} by adapting the proof from
\cite[Section~9]{W12}. 
\begin{proof}[Proof of Corollary~\ref{wooley_corollary}]

We will use the notation \(I^\sharp:=\{1,2,\ldots,k\}\setminus I\) for a
given subset \(I\subseteq \{1,2,\ldots,k\}\). Given \(S\subseteq [N]\), for a
set \(I\subseteq \{1,2,\ldots,k\}\) and for a vector of integers
\(\vec z=(z_i)_{i\in I}\), let \(r_I(\vec z)\) be the number of ordered
\(s\)-tuples \((x_1,\ldots,x_s)\in S^s\) with pairwise distinct entries such that
\[
        \sum_{\ell=1}^s x_\ell^i=z_i,\qquad i\in I.
\]
Let \(G_J(S)\) denote the number of pairs of such ordered \(s\)-tuples with
equal \(i\)-th power sums for every \(i\in J\). Equivalently,
\[
        G_J(S)=\sum_{\vec z}r_J(\vec z)^2.
\]
Since \(r_J\) has total mass \(\gg_s |S|^s\) and support of size
\(O_s(N^{\sum_{i\in J}i})\), Cauchy--Schwarz gives
\begin{equation}\label{eq:GJ-lower}
        G_J(S)\gg_s |S|^{2s}N^{-\sum_{i\in J}i}.
\end{equation}

For any \(\vec z=(z_i)_{i\in J}\), we have
\[
        r_J(\vec z)
        =
        \sum_{\substack{(z_j)_{j\in J^\sharp}\\ 1\le z_j\le sN^j}}
        r_{[k]}(z_1,\ldots,z_k).
\]
Let \(P_{\vec z}\) denote the set of vectors \((z_j)_{j\in J^\sharp}\) such
that \(r_{[k]}(z_1,\ldots,z_k)>0\). Then, by Cauchy--Schwarz,
\begin{align*}
G_J(S)
&\leq
\sum_{\vec z}|P_{\vec z}|
\sum_{\substack{(z_j)_{j\in J^\sharp}\\ 1\le z_j\le sN^j}}
        r_{[k]}(z_1,\ldots,z_k)^2  \\
&\leq
\max_{\vec z}|P_{\vec z}|\cdot G_{[k]}(S),
\end{align*}
where \(G_{[k]}(S)=\sum_{\vec w}r_{[k]}(\vec w)^2\).

Suppose for contradiction that, for every \(\vec z=(z_i)_{i\in J}\), the set
\(P_{\vec z}\) contains no \(m\) vectors
\(\mathbf u_1,\ldots,\mathbf u_m\) satisfying
\[
        u_{h,j}\ne u_{t,j}
        \qquad(1\le h<t\le m,\ j\in J^\sharp).
\]
Then Lemma~\ref{lem:coordinate-packing} gives
\begin{equation}\label{eq:supp}
        \max_{\vec z}|P_{\vec z}|
        \ll_k
        (m-1)\sum_{h\in J^\sharp}N^{(\sum_{j\in J^\sharp}j)-h}
        \ll_{k}
        mN^{(\sum_{j\in J^\sharp}j)-1}.
\end{equation}

Now Corollary~\ref{cor:Vinogradov_in_J}, applied to the unrestricted ordered
solutions, gives
\[
        G_{[k]}(S)
        \ll_{k,\eps} N^{s-k(k+1)/2+\eps}|S|^s
        \ll_{k,\eps} N^\eps |S|^s,
\]
since \(G_{[k]}(S)\) is a subcount of the unrestricted count and
\(s=k(k+1)/2\). Combining this estimate with inequality~\eqref{eq:supp}, we get
\[
        G_J(S)
        \ll_{k,\eps}
        mN^{(\sum_{j\in J^\sharp}j)-1+\eps}|S|^s.
\]
Comparing this with inequality~\eqref{eq:GJ-lower}, and using
\[
        \sum_{i\in J}i+\sum_{j\in J^\sharp}j=s,
\]
we obtain
\[
        |S|^s\ll_{k,\eps}mN^{s-1+\eps}.
\]
Choosing \(\eps<1/(k^2+k+1)\), this contradicts the assumption
\[
        |S|\geq m^{2/k(k+1)}N^{1-2/(k^2+k+1)}
\]
for all sufficiently large \(N\).

Hence there exist some \(\vec z=(z_i)_{i\in J}\) and
\(m\) vectors \(\mathbf u_1,\ldots,\mathbf u_m\in P_{\vec z}\) such that
\[
        u_{h,j}\ne u_{t,j}
        \qquad(1\le h<t\le m,\ j\in J^\sharp).
\]
For each \(1\le h\le m\), choose an ordered \(s\)-tuple
\((x_{1h},\ldots,x_{sh})\in S^s\) with pairwise distinct entries whose full power-sum vector is
\((\vec z,\mathbf u_h)\), and set
\[
        A_h=\{x_{1h},\ldots,x_{sh}\}.
\]
Then \(|A_h|=s\), and the sets \(\{A_h\}_{1\leq h\leq m}\) satisfy the required equalities for
\(j\in J\) and inequalities for \(j\in J^\sharp\). This finishes the proof.
\end{proof}

\subsection{Proof of Corollary \ref{smooth_corollary}}\label{sec:smooth}

We will deduce the corollary from a more technical one, given as follows:

\begin{corollary}\label{product_difference}
Suppose \(0<\theta<1\), \(k\ge1\). Put
\[
        k'=\left\lceil \frac{k}{1-\theta}\right\rceil,
        \qquad
        s=\frac{k'(k'+1)}2 .
\]
Then there is \(N_0=N_0(k,\theta)\) such that, whenever \(N\ge N_0\), for any integer $m\geq 2$ and
\(C\subseteq (N,N+N^\theta]\cap \mathbb Z\) satisfying
\(|C|>m^sN^{\theta(1-1/(s+1))}\), the following holds. For every \(1\le j\le k\),
there exist elements \(x_{i,h}\in C\) with \(1\le i\le s\) and \(1\le h\le m\), such that, writing
\[
        P_h=\prod_{i=1}^s x_{i,h}\qquad(1\le h\le m),
\]
we have, for all \(1\le h<t\le m\),
\[
        |P_h-P_t|
        \in
        [\kappa_1N^{s-j},\kappa_2N^{s-j(1-\theta)}],
\]
where \(\kappa_1,\kappa_2>0\) depend only on \(k\) and \(\theta\).
\end{corollary}

To apply this corollary we first observe that 
since $[N^{(1-\alpha)/r}/2]$ contains at most 
half the elements of $C$, at least half the elements of $C$ are contained in $D := (N^{(1-\alpha)/r}/2, N^{1/r}]$. Next, for $\theta = 1/\sqrt{2r}$
we partition the interval $D$ into subintervals
$$
(M_j,\ M_j+ M_j^\theta],\ j=1,2,\ldots
$$
with $M_1 = N^{(1-\alpha)/r}/2$.  Now, if the $j$-th interval contained at most $M_j^{\theta-\alpha}/4$ many elements of $C$, then the total number of elements of $C$ in all such intervals would be at most
$$
\ll \ \int_{M_1}^{N^{1/r}} {dx \over 4 x^\alpha}\ \ll\ 
{N^{(1-\alpha)/r} \over 4(1-\alpha)}\ <\ {|C|\over 2},
$$
for $0 < \alpha < 1/2$.  Hence, for some choice of $j = j_0$, the interval contains at least $M_{j_0}^{\theta - \alpha}/4$ elements in $C$.  Let $M' = M_{j_0}$ and let $M = \lfloor (M')^r \rfloor$.

Next, assume $r$ has the form $k'(k'+1)/2$ for some
integer $k'$.  We apply Corollary \ref{product_difference} to the interval $(M',M'+(M')^\theta]$ and its intersection with $C$, and with parameters $s=r$, 
$m = \lfloor((M')^{\frac{\theta}{r+1} - \alpha}/4)^{1/r}\rfloor$,
and $j = k$.  From our choice for $\theta$ we note that $k = k'-1$, so we also get $j = k'-1 = \lfloor \sqrt{2r}\rfloor-1$. Since $r\geq 3$, we must have $j\geq 1$. From this
it follows that 
$$
|P_h - P_t| \in [\kappa_1 M^{1- \sqrt{2 \over r} + \frac{1}{r}}, \kappa_2 M^{1 - \sqrt{2\over r} + \frac{3}{r} - \sqrt{\frac{2}{r^3}}}].
$$
Note that $P_h\leq (M'+(M')^\theta)^r\leq 2M$ when $N$ (and hence $M'$) is sufficiently large. Taking $\kappa<\min(\kappa_1,\kappa_2)$ finishes the proof.

We conclude the paper with a proof of Corollary \ref{product_difference}.

\begin{proof}[Proof of Corollary \ref{product_difference}]
Let \(S=C-N=\{c-N:c\in C\}\subseteq [1,N^\theta]\). Since
\[
        1-\frac{1}{s+1}
        =
        1-\frac{2}{k'(k'+1)+2}
        >
        1-\frac{2}{k'(k'+1)+1},
\]
the density assumption is strong enough to apply Corollary~\ref{wooley_corollary}
to \(S\), with \(N^\theta\) in place of \(N\), with \(k'\) in place of \(k\),
and with
\[
        J=\{1,2,\ldots,k'\}\setminus\{j\}.
\]
Thus we obtain \(m\) subsets of \(S\), written as
\[
        \{a_{1,h},\ldots,a_{s,h}\}\qquad(1\le h\le m),
\]
such that their \(r\)-th power sums agree for every \(r\in J\), while their
\(j\)-th power sums are pairwise distinct.

For \(1\le h\le m\), put
\[
        P_h=\prod_{i=1}^s (N+a_{i,h}).
\]
Fix \(1\le h<t\le m\), and put
\[
        D_{h,t}
        =
        (-1)^{j+1}\left(
        \sum_{i=1}^s a_{i,h}^j-\sum_{i=1}^s a_{i,t}^j
        \right).
\]
Then \(D_{h,t}\) is a nonzero integer and
\[
        1\le |D_{h,t}|\le 2sN^{j\theta}.
\]
Using the Taylor expansion for \(\log(1+x)\), we have
\[
\left|\log {P_h\over P_t}\right|
=
\left|
        {D_{h,t}\over jN^j}+E_{h,t}
\right|,
\]
where all terms with \(1\le r\le k'\), \(r\ne j\), vanish, and
\[
        |E_{h,t}|
        \le
        \sum_{r=k'+1}^\infty {2sN^{r\theta}\over rN^r}
        \ll_{k,\theta}
        N^{(k'+1)(\theta-1)}
        < N^{-k-\varepsilon}
\]
for some \(\varepsilon>0\). Hence, for sufficiently large \(N\),
\[
        c_1N^{-j}
        \le
        \left|\log {P_h\over P_t}\right|
        \le
        c_2N^{-j(1-\theta)}.
\]
Since the logarithm tends to \(0\), this gives
\[
        {c_3\over N^j}
        \le
        \left|{P_h\over P_t}-1\right|
        \le
        {c_4\over N^{j(1-\theta)}}.
\]
Finally, since \(P_t\asymp_s N^s\), multiplying by \(P_t\) gives
\[
        |P_h-P_t|
        \in
        [\kappa_1N^{s-j},\kappa_2N^{s-j(1-\theta)}].
\]
Since this holds for every \(1\le h<t\le m\), and since \(N+a_{i,h}\in C\),
the corollary follows.
\end{proof}

\bibliographystyle{abbrv}
\bibliography{main}

\end{document}